\documentclass{amsart}
\usepackage{amssymb,hyperref} % SRC Specials: DVI [Inverse] Search
\newtheorem{thm}{Theorem}[section]

\newtheorem{lem}[thm]{Lemma}

\theoremstyle{definition}

\newtheorem{rem}[thm]{Remark}

\newtheorem{qu}[thm]{Question}

\numberwithin{equation}{section}
\begin{document}

\title[Compact groups]{Compact groups with a set of positive Haar measure satisfying a nilpotent law}%

\author[A. Abdollahi]{Alireza Abdollahi}%
\address{Department of Pure Mathematics, Faculty of Mathematics and Statistics, University of Isfahan, Isfahan 81746-73441, Iran.} 
\email{a.abdollahi@math.ui.ac.ir}%
\author[M. Soleimani Malekan]{Meisam Soleimani Malekan}%
\address{Department of Mathematics, University of Isfahan, Isfahan 81746-73441, Iran; Institute for Research in Fundamental Sciences, School of Mathematics, Tehran, Iran.} 
\email{msmalekan@gmail.com}

%\thanks{The first author is grateful to National Elite Foundation of Iran for financial support.}

%\date{}%
%\dedicatory{}%
%\commby{}%
% ----------------------------------------------------------------
\begin{abstract}
	The following question is proposed in \cite[Question 1.20]{MTVV}.\\
Let $G$ be a compact group, and suppose that ${\mathcal N}_k(G) = \{(x_1,\dots, x_{k+1}) \in G^{k+1} \;|\; 	[x_1, \dots, x_{k+1}] = 1\}$ has positive Haar measure in $G^{k+1}$. Does $G$ have an open $k$-step nilpotent subgroup?\\
The case $k=1$ is already known. We positively answer it for $k=2$. 	
\end{abstract}

\maketitle
% ----------------------------------------------------------------
\section{\bf Introduction and Results}

Let $G$ be a (Hausdorff) compact group. Then $G$ has a unique normalized Haar measure denoted by ${\mathbf m}_G$. The following question is proposed in \cite{MTVV}.

\begin{qu}\cite[Question 1.20]{MTVV} \label{Q-MTVV}
	Let $G$ be a compact group, and suppose that ${\mathcal N}_k(G) = \{(x_1,\dots, x_{k+1}) \in G^{k+1} \;|\; 	[x_1, \dots, x_{k+1}] = 1\}$ has positive Haar measure in $G^{k+1}$. Does $G$ have an open $k$-step nilpotent subgroup?	
\end{qu}

Positive answer to Question \ref{Q-MTVV} is known for $k=1$ (see \cite[Theorem 1.2]{HR}). It follows from  \cite[Theorem 1.19]{MTVV} that Question \ref{Q-MTVV} has positive answer for arbitrary $k$ whenever we further assume that $G$ is totally disconnected i.e. $G$ is a profinite group. 
%
%Theorem 1.19 of \cite{MTVV} answers \cite[Problem 3.1]{S}.
%
Here we answer positively Question \ref{Q-MTVV} for $k=2$ (see Theorem \ref{class2} below).

\section{\bf A preliminary lemma}
We need the following lemma in the proof of our main result. 
\begin{lem}\label{2step}
	Suppose that $G$ is a group and  $x_1,x_2,x_3,g_1,g_2,g_3 \in G$ are  such that 
	\begin{align*}
		1&\overset{1}{=}[x_1,x_2,x_3]\overset{2}{=}[x_1g_1,x_2g_2,x_3g_3]\overset{3}{=}[x_1g_1,x_2g_2,x_3]\overset{4}{=}[x_1g_1,x_2,x_3g_2]\\
		&\overset{5}{=}[x_1g_1,x_2,x_3]\overset{6}{=}[x_1g_1,x_2,x_3g_3]\overset{7}{=}[x_1,x_2,x_3g_1]\overset{8}{=}[x_1,x_2g_2,x_3g_1]\\
		&\overset{9}{=}[x_1,x_2g_2,x_3]\overset{10}{=}[x_1,x_2,x_3g_2]\overset{11}{=}[x_1,x_2g_2,x_3g_3]\overset{12}{=}[x_1,x_2,x_3g_3].
\end{align*}
	Then $[g_1,g_2,g_3]=1$. 
\end{lem}

\begin{proof}
We will throughout  using famous commutator calculus identities.  
\begin{align*}
	1&=[x_1g_1,x_2g_2,g_3]=[[x_1g_1,g_2][x_1g_1,x_2]^{g_2},g_3] \;\; \text{\rm by (2) and (3)}\\
	&=[[x_1g_1,g_2][x_1g_1,x_2],g_3]           \;\; \textbf{\rm by (4) and (5)} \\
	&=[x_1g_1,g_2,g_3]=[[x_1,g_2]^{g_1}[g_1,g_2],g_3] \;\; \textbf{\rm by (5) and (6). \;\; (I)} 
	\end{align*}
On the other hand, 
\begin{align*}
1&=[x_1,x_2g_2,g_1] \;\; \text{\rm by (8) and (9)}\\
&=[[x_1,g_2][x_1,x_2]^{g_2},g_1]=[[x_1,g_2][x_1,x_2],g_1] \;\; \text{by (1) and (10)}\\
&=[x_1,g_2,g_1] \;\; \text{by (1) and (7). \;\; (II)}
\end{align*} 
Also,
\begin{align*}
	1&=[x_1,x_2g_2,g_3] \;\; \text{\rm by (9) and (11)}\\
	&=[[x_1,g_2][x_1,x_2]^{g_2},g_3]=[[x_1,g_2][x_1,x_2],g_3] \;\; \text{by (1) and (10)}\\
	&=[x_1,g_2,g_3] \;\; \text{by (1) and (12). \;\; (III)}
\end{align*}   
Now it follows from   (I), (II) and (III) that $[g_1,g_2,g_3]=1$. 	
\end{proof}
\begin{rem}
	The ``left version" ($g_ix_j$ instead of $x_jg_i$) of Lemma \ref{class2} is not clear to hold. 
	The validity of a similar result to Lemma \ref{class2} for commutators with length more than $3$ is also under question.  
\end{rem}
\section{\bf Compact groups with many elements satisfying the $2$-step nilpotent law}
 
 We need the ``right version" of \cite[Theorem 2.3]{SAE} as follows.
\begin{thm}\label{rv}
If $A$ is a measurable subset with positive
	Haar measure in a compact group $G$, then for any positive integer $k$ there exists an open subset $U$ of $G$ containing $1$ such that 
	$\mathbf{m}_G(A \cap Au_1 \cap \cdots  \cap Au_k) >0$ for all $u_1,\dots,u_k \in U$.
\end{thm} 
\begin{proof}
Since $\mathbf{m}_G(A)=\mathbf{m}_G(A^{-1})$, it follows from Theorem 2.3 of \cite{SAE} that there exists an open subset $V$ of $G$ containing $1$ such that
$$\mathbf{m}_G(A^{-1} \cap v_1A^{-1} \cap \cdots  \cap v_kA^{-1})>0$$
for all $v_1,\dots,v_k \in V$. By \cite[Theorem 4.5]{HS}, there exists an open subset $U\subseteq V$ such that $1\in U$ and $U=U^{-1}$. 
Thus for all $u_1,\dots,u_k \in U$ 
\begin{align*}
0<&\mathbf{m}_G(A^{-1} \cap u_1^{-1}A^{-1} \cap \cdots  \cap u_k^{-1}A^{-1})\\
=&\mathbf{m}_G((A \cap Au_1 \cap \cdots \cap Au_k)^{-1})\\
=& \mathbf{m}_G(A \cap Au_1 \cap \cdots \cap Au_k)
\end{align*}
This completes the proof.
\end{proof}
Now we can prove our main result.
\begin{thm}\label{class2}
		Let $G$ be a compact group, and suppose that ${\mathcal N}_2(G) = \{(x_1,x_2, x_3) \in G\times G \times G \;|\; 	[x_1, x_2, x_3] = 1\}$ has positive Haar measure in $G\times G \times G$. Then $G$ has an open $2$-step nilpotent subgroup.
\end{thm}
\begin{proof}
Let $X:={\mathcal N}_2(G)$. It follows from Theorem \ref{rv} and \cite[Theorem 4.5]{HS}   that  there exists an open subset $U=U^{-1}$ of $G$ containing $1$ such that 
	$$X \cap X\bar{u}_1 \cap \cdots \cap X\bar{u}_{11}\neq \varnothing \;\;\; (*)$$
	for all $\bar{u}_1,\dots,\bar{u}_{11} \in U\times U \times U$. Now take arbitrary elements $g_1,g_2,g_3 \in U$ and consider
			\begin{align*}
	\bar{u}_1&=(g_1^{-1},g_2^{-1},g_3^{-1}), \bar{u}_2=(g_1^{-1},g_2^{-1},1), \bar{u}_3=(g_1^{-1},1,g_2^{-1})\\
	\bar{u}_4 &=(g_1^{-1},1,1),\bar{u}_5=(g_1^{-1},1,g_3^{-1}), \bar{u}_6=(1,1,g_1^{-1}),\bar{u}_7=(1,g_2^{-1},g_1^{-1})\\
	\bar{u}_8 &=(1,g_2^{-1},1), \bar{u}_9=(1,1,g_2^{-1}), \bar{u}_{10}=(1,g_2^{-1},g_3^{-1}), \bar{u}_{11}=(1,1,g_3^{-1}).
	\end{align*}
By $(*)$, there exists $(x_1,x_2,x_3)\in X$ such that all the following 3-tuples are in $X$.	 
	\begin{align*}
	&(x_1g_1,x_2g_2,x_3g_3),(x_1g_1,x_2g_2,x_3),(x_1g_1,x_2,x_3g_2) \\
	&(x_1g_1,x_2,x_3),(x_1g_1,x_2,x_3g_3),(x_1,x_2,x_3g_1),(x_1,x_2g_2,x_3g_1)\\
	&(x_1,x_2g_2,x_3),(x_1,x_2,x_3g_2),(x_1,x_2g_2,x_3g_3),(x_1,x_2,x_3g_3).
	\end{align*}
Now Lemma \ref{2step} implies that $\langle g_1,g_2,g_3 \rangle$ is nilpotent of class at most $2$. Therefore the subgroup $H:=\langle U \rangle$ generated by $U$ is $2$-step nilpotent. Since $H=\bigcup _{n\in {\mathbb N}} U^n$, $H$ is open in $G$. This completes the proof.  
\end{proof}

\end{document}